\documentclass[a4paper,10pt]{article}
\usepackage[utf8]{inputenc}
\usepackage{amsmath,amsfonts}
\usepackage{amsthm}
\usepackage{amssymb}
\usepackage{hyperref}
\usepackage{bm}
\usepackage{multirow}
\usepackage{tikz-cd}
\newtheorem{thm}{Theorem}[section]
\newtheorem{lem}[thm]{Lemma}
\newtheorem{prop}[thm]{Proposition}
\newtheorem{cor}[thm]{Corollary}
\newtheorem*{thm*}{Theorem}

\theoremstyle{definition}
\newtheorem{dfn}[thm]{Definition}
\newtheorem*{dfn*}{Definition}

\newtheorem{obs}[thm]{Remark}
\newtheorem*{ques*}{Question}
\newcommand{\ra}{\rightarrow}

\newcommand{\N}{\mathbb{N}}

\newcommand{\Q}{\mathbb{Q}}
\newcommand{\C}{\mathbb{C}}

\newcommand{\Z}{\mathbb{Z}}

\newcommand{\f}{f_\mathbf{n}}
\newcommand{\Wr}{\mathbb{W}_r}
\newcommand{\WrL}{\mathbb{W}_{r,\mathcal{L}}}

\title{An equivalent condition for $q$-holonomicity}
\author{Giulio Belletti}
\date{}

\begin{document}
	
	\maketitle
	
	\begin{abstract}
		We show that a sequence is $q$-holonomic if and only if it satisfies the elimination property for any subset of variables. The same result also holds for holonomic sequences. As an application, we prove several conjectured closure properties for $q$-holonomic sequences. We also prove that Jones-style sequences for links in any closed $3$-manifold are $q$-holonomic, which in turn implies that the Reshetikhin-Turaev invariants are $q$-holonomic in the colors.
	\end{abstract}

\section{Introduction}

A holonomic sequence $a_n$ in a single variable satisfies a non-trivial linear recurrence relation with polynomial coefficients: $\sum_{i=0}^k P_i(n)a_{n+i}=0$. The $q$-analog concept is a $q$-holonomic sequence: a sequence $f_n$ of elements in $\C(q)$ is $q$-holonomic if it satisfies a non-trivial linear recurrence relation with coefficients polynomials in $q$ and $q^n$: $\sum_{i=0}^k P_i(q,q^n)f_{n+i}=0$.

It is not straightforward to extend this definition to sequences in several variables. Roughly speaking, a sequence in several variables is holonomic (or $q$-holonomic) if it satisfies a maximally overdetermined system of linear recursions like above; the precise definition requires a discussion on the dimension growth of a certain filtration, which will be explained in Section \ref{sec:main}. A $q$-holonomic sequence will in particular satisfy linear recurrence relations with a shift in every individual variable, but in general this is not enough (see for example \cite[Remark 7.3]{GL16}). 

Holonomic sequences were first defined and studied in \cite{Zeil90}. In that paper, $q$-holonomic sequences were also sketched, without the theory being developed; this was done in \cite{Sab93}. Later $q$-holonomic sequences became central in the study of quantum invariants, following the proof that the colored Jones function of a link in $S^3$ is $q$-holonomic \cite{GL05} and the formulation of the AJ conjecture \cite{Gar04}.

The present paper deals with $q$-holonomic sequences, although some of it also applies to holonomic ones (chiefly Theorem \ref{thm:main}).

The gap in sophistication between $q$-holonomic sequences in one variable versus several variables makes it so that many simple properties known in the univariate case were only conjectural in general.

There are two, related reasons to prefer this more abstract definition of $q$-holonomicity, over simply requiring a recurrence relation with a shift in each individual variable (this concept is called $\Wr$-finiteness in \cite{GL16} and is related to $\partial$-finiteness of $\cite{Kou09}$). The first is that a $q$-holonomic sequence is determined by its recurrence relations plus a finite number of initial values; the second is the so-called elimination property, also known as being strongly $\Wr$-finite (see Definition \ref{dfn:stronglywr}). The former property allows proving identities in an algorithmic way; an identity $A=B$ is true if $A$ and $B$ satisfy the same set of equations and if $A=B$ for a certain set of finite values. The latter property is a crucial step in proving that many algorithms terminate (for example, the creative telescoping algorithm of \cite{Zeil91a}). A sequence that satisfies a recurrence relation in each variable will not, in general, satisfy these two properties.

The main result of this paper is the following theorem:

\newtheorem*{thm:main}{Theorem \ref{thm:main}}
\begin{thm:main}
	A sequence $f:\Z^r\ra \C(q)$ is $q$-holonomic if and only if it is strongly $\Wr$-finite.
\end{thm:main}

This bridges the aforementioned gap in sophistication; it means that a sequence is $q$-holonomic if and only if it satisfies a set of recurrence relations with a specific property. This allows us to prove several closure properties for $q$-holonomic sequences, answering conjectures from \cite{GK12} and \cite{GvdV12}. Furthermore it allows us to prove that for any framed link in a closed $3$-manifold, its colored Jones sequence (and thus its Reshetikhin-Turaev invariant) is $q$-holonomic.

In Section \ref{sec:main} we will first give all the necessary definitions and background on $q$-holonomic sequences, then we will prove our main Theorem \ref{thm:main}, and then (in Subsection \ref{ssec:holon}) we will briefly go over the holonomic case. In Section \ref{sec:appl} we will give several applications of Theorem \ref{thm:main}. First, in Subsection \ref{ssec:closure} we will prove several closure properties of $q$-holonomic sequences; finally, in Subsection \ref{ssec:skein} we will give an application to quantum invariants and skein modules.

\textbf{Acknowledgments:} I wish to thank Renaud Detcherry, Stavros Garoufalidis and Christoph Koutschan for interesting conversations. The author was supported by FNRS (grant 1.B.044.25F).

\section{Strongly $\Wr$-finite sequences and $q$-holonomicity}\label{sec:main}

\textbf{Notation:} Throughout this paper, we use boldface letters to indicate tuples; given a $2r$-tuple $\boldsymbol{\eta}=(\eta_1,\dots,\eta_{2r})$, we use $L^\eta M^\eta$ to denote the monomial $L_1^{\eta_1}\cdots L_r^{\eta_r}M_1^{\eta_{r+1}}\cdots M_r^{\eta_{2r}}$, and with $\lvert \eta\rvert$ the sum of the absolute values of all the entries of $\boldsymbol{\eta}$.

Furthermore recall that a sequence of positive numbers $f_N$ is in $O(N^k)$ if there exists a $C$ such that $f_N \leq C N^k$ for all $N$ bigger than some $N_0$, and is in $\Theta(N^k)$ if there are $C_1>0$ and $C_2$ such that $C_1 N^k \leq f_N\leq C_2N^k$ for all $N$ bigger than some $N_0$.

Finally, throughout the paper we follow the definitions and notations of \cite{GL16}, which we point to for any further reading on $q$-holonomic sequences.

\subsection{Equivalence between $q$-holonomicity and strongly $\Wr$-finiteness}

\begin{dfn}\label{dfn:qweyl}
	The \emph{quantum Weyl algebra with $2r$ generators $\Wr$} is obtained by quotienting the non-commutative $\C(q)$-algebra
	
	$$\C(q)[M_1^{\pm1},L_1^{\pm1},\dots, M_r^{\pm1},L_r^{\pm1}]$$
	
	by the ideal generated by 
	
	\begin{gather*}
		L_iM_j=q^{\delta_{ij}}M_jL_i\\
		M_iM_j=M_jM_i\\
		L_iL_j=L_jL_i.
	\end{gather*}.

	The \emph{positive quantum Weyl algebra with $2r$ generators $\mathbb{W}_{r,+}$} is obtained by quotienting the non-commutative $\C(q)$-algebra $\C(q)[L_1,M_1,\dots, L_k,M_k]$ by the same relations.
\end{dfn}

Both these algebras have a filtration given by degree, which we call $\mathcal{F}_N$:

$$\mathcal{F}_N \Wr:=\textrm{span}\langle L^{\boldsymbol{\alpha}}M^{\boldsymbol{\beta}} \textrm{ with } \lvert\alpha\rvert+\lvert\beta\rvert\leq N \rangle$$

and analogously for $\mathbb{W}_{r,+}$.
\begin{dfn}
	If $V$ is a $\C(q)$-vector space, we denote with $S_r(V)$ the space of maps $\Z^r\ra V$, and with $S_{r,+}(V)$ the space of maps $\N^r\ra V$. 
\end{dfn}

We will think of elements of $S_r(V)$ as sequences, and will write $f_{\boldsymbol{n}}$ to mean $f(\boldsymbol{n})$.

It is standard to check that $S_r(V)$ is a $\Wr$-module and $S_{r,+}(V)$ is a $\mathbb{W}_{r,+}$-module, with the action of $L_i$ and $M_i$ defined by 
\begin{align*}
	L_i(f)_{(n_1,\dots,n_k)}:= f_{(n_1,\dots,n_i+1,\dots,n_r)}\\
	M_i(f)_{(n_1,\dots,n_k)}:= q^{n_i}f_{(n_1,\dots,n_r)}.
\end{align*}

\begin{dfn}
	Let $\mathcal{M}$ be a finitely generated $\Wr$-module. By \cite[Proposition 1.5.2]{Sab93}, the dimension of $\mathcal{F}_N \Wr \cdot J$ is eventually equal to a polynomial; the degree of this polynomial is called the \emph{dimension} of $M$, denoted by $d(M)$. It was proven in \cite[Theorem 2.1.1]{Sab93} that if $d(M)<r$, then $M$ is the trivial module. The same definition of dimension can be given for $\mathbb{W}_{r,+}$-modules.
\end{dfn}

\begin{dfn}\label{dfn:qholonomic}
	A $\mathbb{W}_{r,+}$-module $M$ is \emph{$q$-holonomic} if $M$ is finitely generated, contains no monomial torsion, and $d(M)\leq r$.
	
	Similarly, a $\mathbb{W}_{r}$-module $M$ is \emph{$q$-holonomic} if $M$ is finitely generated and $d(M)\leq r$.
	
	An element $f$ in a $\Wr$ or $\mathbb{W}_{r,+}$-module $M$ is \emph{$q$-holonomic} if it spans a $q$-holonomic module. If $f\in S_r(V)$ or $f\in S_{r,+}(V)$, we say $f$ is a $q$-holonomic sequence.
\end{dfn}
Notice that monomials in $\Wr$ act invertibly so a $\Wr$-module cannot have monomial torsion.

Notice furthermore that any $\Wr$-module is also a $\mathbb{W}_{r,+}$-module; in this case the following proposition states that the two concepts of $q$-holonomicity are equivalent.

\begin{prop}\label{prop:Wr+}\cite[Proposition 3.4]{GL16}
	Let $M$ be a $\Wr$-module and $f\in M$. Then $f$ is $q$-holonomic over $\mathbb{W}_{r,+}$ if and only if it is $q$-holonomic over $\Wr$.
\end{prop}

The definition of $q$-holonomicity of a sequence $f\in S_r(V)$ might seem slightly obscure; in practice it means that $f$ satisfies as many recurrence relations as possible without being trivial. 

The following definition provides a possible way to make precise the concept of "satisfying as many recurrence relations as possible".

\begin{dfn}\label{dfn:stronglywr}
	An element $f$ in a $\Wr$-module $M$ is \emph{strongly $\Wr$-finite} if, for any subset $\mathcal{L}\subseteq \{L_1,\dots,L_r,M_1,\dots,M_r\}$ of cardinality exactly $r+1$, there exists a recurrence relation for $f$ in the variables contained in $\mathcal{L}$; in other words, there is an element $P\in \WrL$ such that $P\cdot f=0$.
	
	An element $f$ in a $\Wr$-module $M$ is \emph{strongly $\Wr$-finite with multiplicities} if, for any $2r$-tuple $\boldsymbol{\eta}:=(\alpha_1,\dots,\alpha_r,\beta_1,\dots,\beta_r)\in \N^{2r}$ with support of size exactly $r+1$, there is an element $P\in \mathbb{W}_{r,\mathcal{L}^\eta}$ such that $P\cdot f=0$, with $\mathcal{L}^{\eta}:=\{L_1^{\alpha_1},\dots,L_r^{\alpha_r},M_1^{\beta_1},\dots,M_r^{\beta_r}\}$.
\end{dfn}

We could also introduce strongly $\mathbb{W}_{r,+}$-finite elements, with or without multiplicities, but we omit it for brevity. It is clear though that $f$ in a $\Wr$-module is strongly $\Wr$-finite if and only if it is strongly $\mathbb{W}_{r,+}$-finite (again, with or without multiplicities); we only need to multiply by a large enough monomial to make every exponent non-negative.

The nomenclature of strongly $\Wr$-finite is from \cite{GL16}, while the same notion is called satisfying the elimination property in \cite{Zeil90}; the definition of strongly $\Wr$-finite with multiplicities is to our knowledge new, and is introduced mostly to prove the closure properties of Proposition \ref{prop:rootclosure} and \ref{prop:alphaclosure}. The fact that if $f$ is $q$-holonomic it is strongly $\Wr$-finite can be found in \cite[Theorem 7.2]{GL16} (see also Corollary \ref{cor:equivcond} for a proof), but it turns out that the converse is also true.

\begin{thm}\label{thm:main}
	If $\f\in M$ is strongly $\Wr$-finite, it is $q$-holonomic.
\end{thm}
\begin{proof}
	For simplicity we prove that $\f$ is $q$-holonomic over $\mathbb{W}_{r,+}$; this is enough by Proposition \ref{prop:Wr+}.
	First we fix an ordering $>$ on the variables $L_1,\dots,L_r,M_1,\dots,M_r$; we assume $L_i>M_j$ for all $i,j$, $L_i>L_j$ if $i<j$ and $M_i>M_j$ if $i<j$, but any order will do. We extend this to a monomial ordering first by total degree, and then lexicographically. Second, we say that an $n$-tuple $\eta_1$ \emph{dominates} another $\eta_2$, denoted by $\eta_1\succeq\eta_2$, if each component of $\eta_1$ is larger or equal than the corresponding component of $\eta_2$; we write $\eta_1\nsucceq\eta_2$ if this does not happen. 
	
	Consider the filtration $\mathcal{F}_N\mathbb{W}_{r,+}\cdot f$ and its generating set $\{L^\eta M^\eta\cdot f, \vert\eta\vert\leq N\}$. Given a subset $I\subseteq \{L_1,\dots,L_r,M_1,\dots,M_r\}$ and a tuple $\mathbf{k}=(k_i)_{i\in I}$, denote with $\mathcal{F}_{N,\mathbf{k}}$ the subalgebra of $\mathbb{W}_{r,+}$ generated by $\{L^\eta M^\eta, \vert\eta\vert\leq N, \eta_i < k_i \textrm{ for all } i\in I\}$; it is clear that $\dim(\mathcal{F}_{N,\mathbf{k}})=\Theta(N^{2r-\lvert I\rvert})$, where $\lvert I\rvert$ is the cardinality of $I$.
	
	The statement of the theorem will follow at once from the following lemma:
	
	\begin{lem}\label{lem:aux}
		Suppose $\f$ is strongly $\Wr$-finite, $I\subseteq \{L_1,\dots,L_r,M_1,\dots,M_r\}$ has cardinality $i<r$ and $(k_i)_{i\in I}$ is any tuple of natural numbers. Then, there exists $I_1,\dots,I_{r+1}$ each of cardinality $i+1$ and each containing $I$, and tuples $\mathbf{k^j}:=(k_i^j)_{i^j\in I_j}$ of natural numbers such that $k_i^j=k_i$ if $i\in I$, such that, if $N$ is large enough,
		
		$$\mathcal{F}_{N,\mathbf{k}}\cdot \f\subseteq \sum_{j=1}^{r+1} \mathcal{F}_{N,\mathbf{k^j}}\cdot \f$$
		
		(note that the sum is intended as sets).
	\end{lem}
	
	\begin{proof}
		Consider $J\subseteq \{L_1,\dots,L_r,M_1\dots,M_r\}\setminus I$ of cardinality $r+1$; this exists by assumption on the size of $I$. Because $\f$ is strongly $\Wr$-finite, there must be $P\in \mathbb{W}_{r,J}$ such that $P\cdot\f=0$; we can (up to multiplying by a monomial) assume that all the degrees are non-negative. Call $X$ the largest monomial in $P$ and $\boldsymbol{\xi}$ its multidegree; rescale $P$ so that the coefficient of $X$ is $1$. We claim that $\mathcal{F}_{N,\mathbf{k}}\cdot \f$ is spanned by \begin{equation}\label{elements}\{L^{\mathbf{\eta}}M^\mathbf{\eta}\cdot \f,  \vert\eta\vert\leq N, \eta_i < k_i \textrm{ for all } i\in I\textrm{ and }\boldsymbol{\eta}\nsucceq \boldsymbol{\xi}\}.\end{equation}
		
		Indeed, suppose $Y$ is a monomial in $\mathcal{F}_{N,\mathbf{k}}$ such that $Y\cdot \f$ is not contained in the span of the elements \eqref{elements}, and call $\boldsymbol{\eta}$ its multidegree. We can assume that $Y$ is minimal among monomials with this property. Clearly $\boldsymbol{\eta}\succeq \boldsymbol{\xi}$, otherwise it would be among the generators of the space. Therefore, $Y=\lambda Z X$, where $\lambda\in \C(q)$ and $Z$ is some other monomial, which implies that $Y\cdot \f=\lambda Z (X-P)\cdot \f$ and all the monomials in $Z(X-P)$ must be smaller than $Y$ (and thus belong to \eqref{elements}).
		
		Now, the sets $I_j$ are obtained by adding to $I$ each element of $J$, and the associated tuple $\mathbf{k}^j$ is obtained by adding the degree of $Y$ in the variable $j$. It is then clear that the span of \eqref{elements} is contained in $\sum_{j=1}^{r+1} \mathcal{F}_{N,\mathbf{k^j}}\cdot \f.$
	\end{proof}
	
	Given Lemma \ref{lem:aux}, it is clear that if $\f$ is strongly $\Wr$-finite, $\mathcal{F}_N \mathbb{W}_{r,+}\cdot\f$ is generated by a union of polynomially many (in $r$) sets, each of cardinality $O(N^r)$, and thus $\f$ is $q$-holonomic.
\end{proof}

The following corollary provides the equivalence among the three notions of finiteness considered in this paper.

\begin{cor}\label{cor:equivcond}
	The following are equivalent:
	\begin{enumerate}
		\item \label{item:swrf} $\f$ is strongly $\Wr$-finite;
		\item \label{item:qh} $\f$ is $q$-holonomic;
		\item \label{item:swrfm} $\f$ is strongly $\Wr$-finite with multiplicities.
	\end{enumerate}
\end{cor}
\begin{proof}
	The fact that \eqref{item:swrfm} implies \eqref{item:swrf} is obvious from the definition.
	
	The fact that \eqref{item:qh} implies \ref{item:swrfm} is very similar to the proof of the fact that \eqref{item:qh} implies \eqref{item:swrf} which appears in \cite[Theorem 7.2]{GL16}. Let $\boldsymbol{\eta}\in \Z^{2r}$ with support of size $r+1$; we first notice that $d(\mathbb{W}_{r,\mathcal{L}^\eta})=r+1$, or in other words, $\dim(\mathcal{F}_N\mathbb{W}_{r,\mathcal{L}^\eta})=\Theta(N^{r+1})$. By the $q$-holonomicity assumption, $\dim(\mathcal{F}_N\Wr\cdot f)=O(N^r)$, which means that for $N$ big enough, there must be a non-zero $P\in \mathcal{F}_N\mathbb{W}_{r,\mathcal{L}^\eta}$ such that $P\cdot \f=0$.
	
	Finally, the fact that \eqref{item:swrf} implies \eqref{item:qh} is the content of Theorem \ref{thm:main}.
	
\end{proof}

\subsection{Strongly $A_r$-finite sequences and holonomicity}\label{ssec:holon}

In this subsection we repeat the previous constructions, definitions and proofs for the classical case of holonomic sequences, introduced in \cite{Zeil90}. For simplicity we only look at sequences indexed by natural numbers, rather than integers, with the understanding that everything carries over without any problem to the $\Z$ case. We will mainly use these results to prove Proposition \ref{prop:evalclosure}, stating that evaluations of $q$-holonomic sequences at roots of unity give holonomic sequences.

\begin{dfn}\label{dfn:weylalg}
	The Weyl algebra $A_r$ is the free commutative algebra generated by $l_1,\dots,l_r,m_1,\dots,m_r$ with the relations $l_im_j-m_jl_i=\delta_{ij}$, $m_im_j-m_jm_i=0$ and $l_il_j-l_jl_i=0$. The set of sequences $S_{r,+}(\C):=\{\N^r\ra \C\}$ is an $A_r$-module, with action given by $(m_j\cdot f)_\mathbf{n}=n_j f_{\mathbf{n}}$ and by $(l_i\cdot f)_\mathbf{n}=f_{\mathbf{n}+\delta_i}$.
\end{dfn}
	
	We can define a filtration $\mathcal{F}_N$ on $A_r$ exactly as before by taking polynomials in the generators of total degree less than or equal to $N$.
	
\begin{dfn}\label{dfn:holonomic} We say that an $A_r$-module $M$ is \emph{holonomic} if $M$ is finitely generated as an $A_r$-module and if $dim(\mathcal{F}_NA_r\cdot J)=O(N^r)$ for a finite generating set $J$ of $M$. A sequence $f$ is holonomic if it spans a holonomic cyclic module in $S_r(\C)$.\end{dfn}
	
	\begin{dfn}
		An element $f$ in a $A_r$-module $M$ is \emph{strongly $A_r$-finite} (or \emph{satisfies the elimination property}) if, for any subset $\mathcal{L}\subseteq \{l_1,\dots,l_r,m_1,\dots,m_r\}$ of cardinality exactly $r+1$, there exists a recurrence relation for $f$ in the variables contained in $\mathcal{L}$; in other words, there is an element $P\in A_r$ only involving the variables in $\mathcal{L}$ such that $P\cdot f=0$.
\end{dfn}

The proof of the following theorem is exactly the same as the proof of Theorem \ref{thm:main} and Corollary \ref{cor:equivcond}.

\begin{thm}\label{thm:holonomic}
	A sequence $f\in S_{r,+}(\C)$ is strongly $A_r$-finite if and only if it is holonomic.
\end{thm}
	
The fact that a holonomic sequence is strongly $A_r$-finite is well known; it was first proved in \cite[Lemma 4.1]{Zeil90} in the differential case (the discrete case follows at once, see for example \cite[Proposition 26]{CK19}). The converse is, to the best of our knowledge, a new result.
\section{Applications}\label{sec:appl}

In this section we provide some applications for Theorem \ref{thm:main}. The first set of applications will be to prove some closure properties for $q$-holonomic sequences; the proofs will essentially be straightforward generalizations of the analogous proofs for sequences in one variable, made possible by Theorem \ref{thm:main}. The second application will be proving that certain sequences of elements in skein modules are $q$-holonomic.	

\subsection{Closure properties for strongly $\Wr$-finite sequences}\label{ssec:closure}
The univariate case of the following two closure properties appears in \cite{GK12} as Theorems 1 and 3; the multivariate case follows by similar reasoning after applying Theorem \ref{thm:main}. This answers Conjecture 5 in the same paper.

\begin{prop}\label{prop:rootclosure}
	Suppose $f(q)$ is a $q$-holonomic sequence of rational functions in $q$, and $\omega$ is a root of unity; then $f_\omega:=f(\omega q)$ is also $q$-holonomic.
\end{prop}
\begin{proof} Suppose $\omega$ is a $p$-th root of unity; then $M_i^p f_\omega=(M_i^p f)(\omega q)$ for all $i$ and, trivially, $L_i f_\omega= (L_if)(\omega q)$, therefore if $P\cdot f=0$ and all $M_i$ powers in $P$ are divisible by $p$, $P(\omega q)\cdot f_\omega=0$. 
	
	Given any subset $\mathcal{L}$ of $\{L_1,\dots,L_r,M_1,\dots,M_r\}$ we need to find a recurrence relation for $f_\omega$; since $f$ is $q$-holonomic, by Corollary \ref{cor:equivcond} there is a $P\in \WrL$ such that all powers of $M_i$ in $P$ are divisible by $p$; this is the desired recursion.
	
\end{proof}

\begin{prop}\label{prop:alphaclosure}
	Suppose $f(q)$ is $q$-holonomic and $\alpha\in\Q$; then $f(q^\alpha)$ is also $q$-holonomic.
\end{prop}

Note that the sequence $f(q^\alpha)$ belongs to $S_r(\Q(q^{\alpha}))$, which is itself a $\Wr$-module.
\begin{proof}
	Let $k$ be the denominator of $\alpha$ in reduced form and let $\mathcal{L}$ be a subset of the variables of cardinality $r+1$. If there is a $P$ in the variables $\mathcal{L}$ annihilating $f$ that is a polynomial in the variables $q^k$ and $M_i^k$ for every $M_i\in \mathcal{L}$, then $P(q^\alpha)$ annihilates $f(q^\alpha)$ and therefore $f(q^\alpha)$ is $q$-holonomic.
	
	To show this, we look at the filtration of $\Wr$ given by $\textrm{span}_{\C(q^k)}\langle L^\alpha M^{k\beta}\rangle$, with the understanding that any variable not in $\mathcal{L}$ will have null exponent. The dimension of this grows like $O(N^{r+1})$. Furthermore the dimension of $\mathcal{F}_Nf$ grows like $O(N^r)$ also as a module over $\C(q^k)$, since $\C(q)$ is an extension of $\C(q^k)$ of finite degree equal to $k$. Then just as in the proof of Corollary \ref{cor:equivcond} we can find a recurrence relation for $f$ with the required characteristics.
	
\end{proof}

We denote by $D_q$ the operator on $S_r(\C(q))$ acting by derivative in $q$; in other words, $(D_q\cdot f)_{\mathbf{n}}:= \frac{d}{dq}f_\mathbf{n}$.

The following two closure properties in the univariate case appear in \cite{GvdV12} as Theorem 1.4 and 1.5, and in the subsequent remark they are conjectured to hold in general. Once again, the proof of the multivariate case follows a similar reasoning as the univariate case after applying Theorem \ref{thm:main}
\begin{prop}\label{prop:diffclosure}
	Suppose $f$ is $q$-holonomic; then $D_q f$ is also $q$-holonomic.
\end{prop}
\begin{proof}
	Consider any monomial $q^kL^\mathbf{\alpha}M^\mathbf{\beta}$. A straightforward calculation shows that $D_q \cdot \left(q^kL^\mathbf{\alpha}M^\mathbf{\beta}\cdot f\right)_\mathbf{n}=D_q\cdot q^k q^{\mathbf{\beta}\cdot\mathbf{n}}f_{\mathbf{n}+\mathbf{\alpha}}$ is equal to $$q^k q^{\mathbf{\beta}\cdot\mathbf{n}}D_q\cdot f_{\mathbf{n}+\mathbf{\alpha}}+\left(k+\sum_{j=1}^r n_j\beta_j\right)q^{k+\mathbf{\beta}\cdot\mathbf{n}-1}f_{\mathbf{n}}$$ which is in turn equal to $$\left(q^k L^\alpha M^\beta\cdot D_q f\right)_\mathbf{n}+\left(k+\sum_{j=1}^r n_j\beta_j\right)q^{k-1}L^\alpha M^\beta f_{\mathbf{n}}$$.
	
	Therefore, given a subset $I\subseteq \mathcal{L}$ and a recurrence relation $P$ in those variables, we have that $0=D_q (P \cdot f)_\mathbf{n}=(P\cdot D_qf)_\mathbf{n}+R$ where $R$ is a linear combination of sequences of the form $\left(k+\sum_{j=1}^r n_j\beta_j\right)q^{k-1}L^\alpha M^\beta f_{\mathbf{n}}$. Standard closure properties for $q$-holonomic sequences imply that each of these sequences is $q$-holonomic, therefore there must be a recurrence relation for $R$ in the variables $I$. Multiplying this recurrence relation by $P$ on the right provides a recurrence relation for $D_qf$ in the variables $I$, proving that $D_qf$ is strongly $\Wr$-finite and hence $q$-holonomic.
\end{proof}
\begin{prop}\label{prop:evalclosure}
	Suppose $f$ is $q$-holonomic and $\omega$ is a root of unity; then $f(\omega)$ is a holonomic sequence.
\end{prop}
\begin{proof}
	By Proposition \ref{prop:rootclosure} we only need to prove this for $\omega=1$. We can assume that $f(1)_\mathbf{n}\neq 0$ for arbitrarily large $\mathbf{n}$, otherwise the statement is trivially true (since a sequence with finite support is always holonomic).
	
	 Fix a subset of the variables $I$ and $P\in \mathbb{W}_{r,I}$ such that $P\cdot f=0$. If $P$ becomes a non-zero polynomial in the $L$ variables after substituting $q=1$ and $M_i=1$ for all $i$'s, then $f(1)$ satisfies a recurrence relation in the variables $I$ (actually, just the $L$ variables among those) and we have found our desired recurrence. 
	 
	 If not, we can take the (trivial) sequence $P\cdot f$ and apply $D_q$ to it. We saw in the proof of Proposition \ref{prop:diffclosure} that $$(D_q\cdot q^kL^\alpha M^\beta f)_\mathbf{n}=(q^kL^\alpha M^\beta\cdot D_qf)_\mathbf{n}+\left(k+\sum_{j=1}^r m_j\beta_j\right)q^{k-1}L^\alpha M^\beta f_{\mathbf{n}}$$
	where $m_j$ is the operator in $A_r$ of Definition \ref{dfn:weylalg}.
	Therefore $0=D_q (P \cdot f)_\mathbf{n}=(P\cdot D_qf)_\mathbf{n}+(R\cdot f)_\mathbf{n}$ where $R$ is an operator in the variables $I$ plus an extra $m_j$ for each $M_j$ in $I$. We can once again evaluate this sequence at $q=1, M_j=1$; the first summand is equal to $0$ by assumption, so if $R\neq 0$ when evaluated at $q=1, M_j=1$, we have a recurrence for $f(1)$ in the same variables $I$ with $m_j$ replacing $M_j$ for all $j$. If not, we can keep repeating this process by taking $q$-differentials; at some point we must obtain a non-zero recurrence relation. 
		
	To see this is the case, fix an $\alpha$ among the $L$-degrees of $P$ and consider the polynomial $P_\alpha$ of its coefficients in $q$ and $M$. Pick an $\mathbf{n}$ such that $f(1)_\mathbf{n}\neq 0$ and look at the polynomial in $q$ given by $P_\alpha\cdot f_\mathbf{n-\alpha}$: it must vanish in $1$ to order $s$, or in other words, it is equal to $\lambda_\alpha (q-1)^s+O((q-1)^{s+1})$. Then $D_q^s P_\alpha\cdot f_\mathbf{n-\alpha}$ must be non-zero when evaluated in $1$. This shows that in the above formula, after $s$ steps we must obtain a recurrence relation with at least one non-zero term.
	
\end{proof}

\subsection{Application to skein modules}\label{ssec:skein}

Throughout this section when we write $A$ we assume that $q=A^2$; therefore any $\C(A)$-module is also a $\C(q)$-module (and the same holds for $\Z[A,A^{-1}]$).

Skein modules were introduced independently in \cite{Prz1} and \cite{Tur}; for their definition and basic properties see for example \cite{Prz99}. We will also look at Reshetikhin-Turaev invariants \cite{RT91}; for an introduction and definitions we refer the reader to \cite{Lic97}.

Let $N$ be a closed oriented $3$-manifold and let $Sk(N)$ be its Kauffman bracket skein module with $\Z[A,A^{-1}]$ coefficients. Let $S_{r}(Sk(N))$ be the module of sequences $\Z^r\ra Sk(N)$, and let $S_r^{sk}(N)$ be $S_{r}(Sk(N))\otimes \C(A)$, where the tensor product is of $\Z[A,A^{-1}]$-modules. Then $S_r^{sk}(N)$ is a $\mathbb{W}_{r}$-module, with $q$ acting like multiplication by $A^2$. 

Given a framed knot $K\subseteq N$ and a polynomial $P\in \Z[A,A^{-1}][x]$, we can define the element $P(K)\in Sk(N)$ first by defining $K^n$ as $n$ parallel copies of $K$, and the rest linearly. Similarly given a framed link with $r$ ordered components $K\subseteq N$ and $r$ polynomials $\mathbb{P}:=(P_1,\dots,P_r)$, we can define the  element $P(K)\in Sk(N)$ by doing this process on each component and expanding multilinearly.

\begin{dfn}
		Let $N$ be a compact oriented $3$-manifold and $K\subseteq N$ be a framed link with $r$ components. Given an $r$-tuple $\mathbf{n}=(n_1,\dots,n_r)\in \Z^r$ we can define the \emph{colored Jones skein} of $K$, denoted with $J_{\mathbf{n}}^K$, as $(-1)^{n_1+\dots+n_r-r}T(K)$ where $T$ is the $r$-tuple $(T_{n_1-1},\dots,T_{n_r-1})$, and $T_i$ is the $i$-th Chebyshev polynomial of the second kind, defined by the recurrence relations
		
			$$\begin{cases}
			T_n(x)=x T_{n-1}(x)-T_{n-2}(x)\\
			T_1(x)=x\\
			T_0(x)=1
		\end{cases}
		$$
\end{dfn}

Let $\mathcal{U}$ be the set of roots of unity, and $\C^{\mathcal{U}}$ the algebra of (set) functions from $\mathcal{U}$ to $\C$. The set of sequences in $\C^{\mathcal{U}}$ is not a $\C(q)$-vector space (and thus not a $\Wr$-module), but merely a $\C[q,q^{-1}]$ module, with $q$ acting by the identity map. Notice that the identity map $\mathcal{U}\ra \C$ is not the identity of the algebra $\C^{\mathcal{U}}$, which is the map sending every root of unity to $1$. We can turn $S_r(\C^\mathcal{U})$ into a $\C(q)$-vector space and a $\Wr$-module $\widehat{S_r(\C^{\mathcal{U}})}$ by tensoring it with $\C(q)$. 

We could also have proceeded by turning $\C^{\mathcal{U}}$ into a $\C(q)$-vector space directly, by tensoring it in the same manner. We denote this by $\widehat{\C^{\mathcal{U}}}$. The difference between the two objects is that a sequence $f_{\mathbf{n}}$ in $\widehat{S_r(\C^{\mathcal{U}})}$ is trivial if and only if there is a finite set in $\mathcal{U}$ such that for every $\mathbf{n}$, $f_{\mathbf{n}}=0$ outside of this set, whereas a sequence in $S_r(\widehat{C^\mathcal{U}})$ is trivial if and only if for every $\mathbf{n}$, $f_{\mathbf{n}}=0$ outside a finite set.

\begin{dfn}
	Let $K\subseteq N$ be a framed, $r$-component link. Then the sequence $RT(M,L,\mathbf{n})\in \widehat{S_r(\C^\mathcal{U})}$ is defined as 
	
	$$f_\mathbf{n}^{M,L}(\zeta):= (-1)^{|\mathbf{n}-\mathbf{1}|}RT_r(M,L,\mathbf{n}-\mathbf{1},\zeta)$$
	where $RT_r(M,L,\mathbf{n}-\mathbf{1},\zeta)$ is the relative Reshetikhin-Turaev invariant of the pair $(M,L)$ with colors $\mathbf{n}-\mathbf{1}$ (component-wise) at the $r$-th root of unity $\zeta$. This definition works if all components of $\mathbf{n}$ are positive; it can be extended oddly in the general case. 
\end{dfn}
\begin{obs}
	The sign and the shift by $1$ in the indices defining $J_{\mathbf{n}}^K$ is done purely to have the notation agree with the notation usually employed in the statement of the AJ conjecture.
\end{obs}
\begin{thm}\label{thm:skeinqholonomic}
	For any closed manifold $N$ and any $r$-component framed link $K\subseteq N$, the sequence $J_\mathbf{n}^K\in S_r^{sk}(N)$ is $q$-holonomic.
\end{thm}
\begin{proof}
	Take $\mu_1,\dots,\mu_r$ the meridians and $\lambda_1,\dots,\lambda_r$ the longitudes (i.e. the parallel copies of each component) of $K$. These are generators for the algebra $Sk(\partial N_K)$ where $N_K$ is the exterior of $K$. The skein module $Sk(N)$ is a $Sk(\partial N)$-module, and there is a bilinear map $\langle\cdot,\cdot\rangle: Sk(\sqcup_r S^1\times D^2)\times Sk(N_K)\ra Sk(N)$, given by gluing and inclusion. In this context, $J_\mathbf{n}^K=\langle J_\mathbf{n}^c,\varnothing\rangle$ where $c$ is the link in $\sqcup_r S^1\times D^2$ comprising of all the core curves. Some standard calculations give that $$\langle J_\mathbf{n}^c\cdot \mu_i,\varnothing\rangle=(-A^{2n_i}-A^{-2n_i})\langle J_\mathbf{n}^c,\varnothing\rangle$$ and $$\langle J_\mathbf{n}^c\cdot \lambda_i,\varnothing\rangle=\langle -J_{\mathbf{n}+\delta_i}^c-J_{\mathbf{n}-\delta_i}^c,\varnothing\rangle.$$
	
	Notice that in standard references (for example in Section 9 of \cite{KL94}) these formulas are proved with positive colors, but the extension to negative colors is straightforward since $T_{-n}=-T_n$. Because of \cite[Corollary 1.7]{BD25}, for any choice of $r+1$ curves among the $\lambda_i$s and the $\mu_j$s, there is an ordered polynomial $P$ in those curves that is non-zero in $Sk(\partial N_K)$ but such that $P\cdot \varnothing=0$ in $Sk(N_K)$. Then this gives a recurrence relation for $J_\mathbf{n}^L$ in the variables corresponding to the curves, proving that the sequence is strongly $\Wr$-finite and thus $q$-holonomic by Theorem \ref{thm:main}.
	
\end{proof}

Theorem \ref{thm:skeinqholonomic} in particular provides a new proof that the colored Jones polynomial of a link in $S^3$ is $q$-holonomic, first proved in \cite{GL05} (this new proof in the case for knots was already implicit in \cite[Corollary 1.7]{BD25}). An extension of the AJ conjecture to any manifold will be explored in future work.

\begin{cor}\label{cor:rtqhol}
	For any framed link $K$ in a closed oriented $3$-manifold $N$, the sequence $RT(N,K,\mathbf{n})$ is $q$-holonomic.
\end{cor}
\begin{proof}
	The Gilmer-Masbaum moment map associates to any skein in $Sk(N)$ the function in $\C^\mathcal{U}$ sending a root of unity to the Reshetikhin-Turaev invariant of the skein evaluated at that root of unity. Applying the map to each value in a sequence induces a map of $\Wr$-modules $S_r^{sk}(N)\ra \widehat{S_r(\C^\mathcal{U})}$. This immediately proves that $RT(N,K,\mathbf{n})$ is strongly $\Wr$-finite, since any non-zero recurrence in $S_r^{sk}(N)$ is sent to a non-zero recurrence in $S_r(\C^{\mathcal{U}})$ in the same variables (alternatively, the image of a $q$-holonomic module is $q$-holonomic).
\end{proof}
\begin{obs}
	Theorem \ref{thm:skeinqholonomic} works just as well when considering $J_\mathbf{n}^K$ as an element of $S_r(Sk(N,\Q(A)))$ instead; in this case Corollary \ref{cor:rtqhol} says that $RT(N,K,\mathbf{n})$ is $q$-holonomic as a sequence in $S_r(\widehat{\C^\mathcal{U}})$. 
\end{obs}

\bibliographystyle{alpha}
\bibliography{biblio}

@article{bd25,
  title={An effective proof of finiteness for Kauffman bracket skein modules},
  author={Belletti, Giulio and Detcherry, Renaud},
  journal={arXiv preprint arXiv:2507.02589},
  year={2025}
}

@phdthesis{Kou09,
 author = {Christoph Koutschan},
 title = {Advanced applications of the holonomic systems approach},
 school = {Research Institute for Symbolic Computation (RISC), Johannes Kep\-ler University},
 address = {Linz, Austria},
 year = {2009},
 url = {https://risc.jku.at/sw/holonomicfunctions/},
}

@book {KL94,
    AUTHOR = {Kauffman, Louis H. and Lins, S\'ostenes L.},
     TITLE = {Temperley-{L}ieb recoupling theory and invariants of
              {$3$}-manifolds},
    SERIES = {Annals of Mathematics Studies},
    VOLUME = {134},
 PUBLISHER = {Princeton University Press, Princeton, NJ},
      YEAR = {1994},
     PAGES = {x+296},
      ISBN = {0-691-03640-3},
   MRCLASS = {57N10 (05C90 57M25 82B23)},
  MRNUMBER = {1280463},
MRREVIEWER = {David\ N.\ Yetter},
       DOI = {10.1515/9781400882533},
       URL = {https://doi.org/10.1515/9781400882533},
}

@article {RT91,
    AUTHOR = {Reshetikhin, N. and Turaev, V. G.},
     TITLE = {Invariants of {$3$}-manifolds via link polynomials and quantum
              groups},
   JOURNAL = {Invent. Math.},
  FJOURNAL = {Inventiones Mathematicae},
    VOLUME = {103},
      YEAR = {1991},
    NUMBER = {3},
     PAGES = {547--597},
      ISSN = {0020-9910,1432-1297},
   MRCLASS = {57N10 (17B37 57M25 81R50)},
  MRNUMBER = {1091619},
MRREVIEWER = {Louis\ H.\ Kauffman},
       DOI = {10.1007/BF01239527},
       URL = {https://doi.org/10.1007/BF01239527},
}

@inproceedings {GK12,
    AUTHOR = {Garoufalidis, Stavros and Koutschan, Christoph},
     TITLE = {Twisting {$q$}-holonomic sequences by complex roots of unity},
 BOOKTITLE = {I{SSAC} 2012---{P}roceedings of the 37th {I}nternational
              {S}ymposium on {S}ymbolic and {A}lgebraic {C}omputation},
     PAGES = {179--186},
 PUBLISHER = {ACM, New York},
      YEAR = {2012},
      ISBN = {978-1-4503-1269-1},
   MRCLASS = {68W30 (12E05)},
  MRNUMBER = {3206302},
       DOI = {10.1145/2442829.2442857},
       URL = {https://doi.org/10.1145/2442829.2442857},
}

@article {GvdV12,
    AUTHOR = {Garoufalidis, Stavros and van der Veen, Roland},
     TITLE = {Asymptotics of quantum spin networks at a fixed root of unity},
   JOURNAL = {Math. Ann.},
  FJOURNAL = {Mathematische Annalen},
    VOLUME = {352},
      YEAR = {2012},
    NUMBER = {4},
     PAGES = {987--1012},
      ISSN = {0025-5831,1432-1807},
   MRCLASS = {57M27 (57M25)},
  MRNUMBER = {2892459},
MRREVIEWER = {Francesco\ Costantino},
       DOI = {10.1007/s00208-011-0662-3},
       URL = {https://doi.org/10.1007/s00208-011-0662-3},
}

@book {Lic97,
    AUTHOR = {Lickorish, W. B. Raymond},
     TITLE = {An introduction to knot theory},
    SERIES = {Graduate Texts in Mathematics},
    VOLUME = {175},
 PUBLISHER = {Springer-Verlag, New York},
      YEAR = {1997},
     PAGES = {x+201},
      ISBN = {0-387-98254-X},
   MRCLASS = {57M25 (57N10)},
  MRNUMBER = {1472978},
MRREVIEWER = {Darryl\ McCullough},
       DOI = {10.1007/978-1-4612-0691-0},
       URL = {https://doi.org/10.1007/978-1-4612-0691-0},
}

@article {Prz1,
    AUTHOR = {Przytycki, J\'{o}zef H.},
     TITLE = {Skein modules of {$3$}-manifolds},
   JOURNAL = {Bull. Polish Acad. Sci. Math.},
  FJOURNAL = {Bulletin of the Polish Academy of Sciences. Mathematics},
    VOLUME = {39},
      YEAR = {1991},
    NUMBER = {1-2},
     PAGES = {91--100},
      ISSN = {0239-7269},
   MRCLASS = {57M25 (57N10)},
  MRNUMBER = {1194712},
}

@incollection {Sab93,
    AUTHOR = {Sabbah, Claude},
     TITLE = {Syst\`emes holonomes d'\'equations aux {$q$}-diff\'erences},
 BOOKTITLE = {{$D$}-modules and microlocal geometry ({L}isbon, 1990)},
     PAGES = {125--147},
 PUBLISHER = {de Gruyter, Berlin},
      YEAR = {1993},
      ISBN = {3-11-012959-0},
   MRCLASS = {33D60 (32C38 39A10)},
  MRNUMBER = {1206016},
}

@article {Zeil90,
    AUTHOR = {Zeilberger, Doron},
     TITLE = {A holonomic systems approach to special functions identities},
   JOURNAL = {J. Comput. Appl. Math.},
  FJOURNAL = {Journal of Computational and Applied Mathematics},
    VOLUME = {32},
      YEAR = {1990},
    NUMBER = {3},
     PAGES = {321--368},
      ISSN = {0377-0427,1879-1778},
   MRCLASS = {33C20 (05A10 33D20 68Q40)},
  MRNUMBER = {1090884},
MRREVIEWER = {R.\ A.\ Askey},
       DOI = {10.1016/0377-0427(90)90042-X},
       URL = {https://doi.org/10.1016/0377-0427(90)90042-X},
}

@inproceedings {Gar04,
    AUTHOR = {Garoufalidis, Stavros},
     TITLE = {On the characteristic and deformation varieties of a knot},
 BOOKTITLE = {Proceedings of the {C}asson {F}est},
    SERIES = {Geom. Topol. Monogr.},
    VOLUME = {7},
     PAGES = {291--309},
 PUBLISHER = {Geom. Topol. Publ., Coventry},
      YEAR = {2004},
   MRCLASS = {57M27 (57M25 57N10)},
  MRNUMBER = {2172488},
MRREVIEWER = {R\u azvan\ Gelca},
       DOI = {10.2140/gtm.2004.7.291},
       URL = {https://doi.org/10.2140/gtm.2004.7.291},
}

@article {Zeil91a,
    AUTHOR = {Zeilberger, Doron},
     TITLE = {The method of creative telescoping},
   JOURNAL = {J. Symbolic Comput.},
  FJOURNAL = {Journal of Symbolic Computation},
    VOLUME = {11},
      YEAR = {1991},
    NUMBER = {3},
     PAGES = {195--204},
      ISSN = {0747-7171,1095-855X},
   MRCLASS = {33C20},
  MRNUMBER = {1103727},
MRREVIEWER = {R.\ A.\ Askey},
       DOI = {10.1016/S0747-7171(08)80044-2},
       URL = {https://doi.org/10.1016/S0747-7171(08)80044-2},
}

@article {CK19,
    AUTHOR = {Chen, Shaoshi and Koutschan, Christoph},
     TITLE = {Proof of the {W}ilf-{Z}eilberger conjecture for mixed
              hypergeometric terms},
   JOURNAL = {J. Symbolic Comput.},
  FJOURNAL = {Journal of Symbolic Computation},
    VOLUME = {93},
      YEAR = {2019},
     PAGES = {133--147},
      ISSN = {0747-7171,1095-855X},
   MRCLASS = {33D05},
  MRNUMBER = {3913568},
MRREVIEWER = {Nancy\ Shanshan\ Gu},
       DOI = {10.1016/j.jsc.2018.06.003},
       URL = {https://doi.org/10.1016/j.jsc.2018.06.003},
}

@article {GL16,
    AUTHOR = {Garoufalidis, Stavros and L\^{e}, Thang T. Q.},
     TITLE = {A survey of {$q$}-holonomic functions},
   JOURNAL = {Enseign. Math.},
  FJOURNAL = {L'Enseignement Math\'{e}matique},
    VOLUME = {62},
      YEAR = {2016},
    NUMBER = {3-4},
     PAGES = {501--525},
      ISSN = {0013-8584},
   MRCLASS = {57M27 (33D15 33F10 39A13)},
  MRNUMBER = {3692896},
MRREVIEWER = {Gw\'{e}na\"{e}l Massuyeau},
       DOI = {10.4171/LEM/62-3/4-7},
       URL = {https://doi-org.proxy-bu1.u-bourgogne.fr/10.4171/LEM/62-3/4-7},
}

@article {Tur,
    AUTHOR = {Turaev, V. G.},
     TITLE = {The {C}onway and {K}auffman modules of a solid torus},
   JOURNAL = {Zap. Nauchn. Sem. Leningrad. Otdel. Mat. Inst. Steklov.
              (LOMI)},
  FJOURNAL = {Zapiski Nauchnykh Seminarov Leningradskogo Otdeleniya
              Matematicheskogo Instituta imeni V. A. Steklova Akademii Nauk
              SSSR (LOMI)},
    VOLUME = {167},
      YEAR = {1988},
    NUMBER = {Issled. Topol. 6},
     PAGES = {79--89, 190},
      ISSN = {0373-2703},
   MRCLASS = {57M25},
  MRNUMBER = {964255},
MRREVIEWER = {J\'{o}zef H. Przytycki},
       DOI = {10.1007/BF01099241},
       URL = {https://doi-org.proxy-bu1.u-bourgogne.fr/10.1007/BF01099241},
}

@article {GL05,
    AUTHOR = {Garoufalidis, Stavros and L\^e, Thang T. Q.},
     TITLE = {The colored {J}ones function is {$q$}-holonomic},
   JOURNAL = {Geom. Topol.},
  FJOURNAL = {Geometry and Topology},
    VOLUME = {9},
      YEAR = {2005},
     PAGES = {1253--1293},
      ISSN = {1465-3060,1364-0380},
   MRCLASS = {57M27 (33D15 33F10 57M25 57N10)},
  MRNUMBER = {2174266},
MRREVIEWER = {Vladimir\ V.\ Tchernov},
       DOI = {10.2140/gt.2005.9.1253},
       URL = {https://doi.org/10.2140/gt.2005.9.1253},
}

@article {Prz99,
    AUTHOR = {Przytycki, J\'ozef H.},
     TITLE = {Fundamentals of {K}auffman bracket skein modules},
   JOURNAL = {Kobe J. Math.},
  FJOURNAL = {Kobe Journal of Mathematics},
    VOLUME = {16},
      YEAR = {1999},
    NUMBER = {1},
     PAGES = {45--66},
      ISSN = {0289-9051},
   MRCLASS = {57M25},
  MRNUMBER = {1723531},
}
\end{document}